\documentclass[10pt,a4paper,reqno]{amsart}
\usepackage{amsfonts,amsthm,latexsym,amsmath,amssymb,amscd,amsmath, epsf}

\newtheorem{theorem}{Theorem}

\newtheorem{corollary}{Corollary}
\newtheorem{proposition}{Proposition}

\newtheorem{problem} {Problem}

\newcommand {\bC} {\mathbb {C}}

\newcommand {\supp} {\mathrm sup~}
\newcommand {\HH} {\mathcal H}

\newcommand{\PP}{\mathbb P}
\newcommand{\w}{\widetilde w}

\begin{document}
             \numberwithin{equation}{section}

             \title[On the Waring problem for polynomial rings]
             {On the Waring problem for polynomial rings}

\author[R.~Fr\"oberg]{Ralf Fr\"oberg}
\address{Department of Mathematics, Stockholm University, SE-106 91, Stockholm,
            Sweden}
\email{ralff@math.su.se}

\author[G.~Ottaviani] {Giorgio Ottaviani}
\address{Universit\'a degli Studi di Firenze, Dipartimento di Matematica "U. Dini", Viale Morgagni, 67a, I Ð 50134 Firenze}Ê
\email{ottavian@math.unifi.it}

\author[B.~Shapiro]{Boris Shapiro}
\address{Department of Mathematics, Stockholm University, SE-106 91, Stockholm,
            Sweden}
\email{shapiro@math.su.se}

\date{\today}

\begin{abstract}
In this note we discuss an  analog of the classical Waring  problem for $\bC[x_0, x_1,...,x_n]$. Namely, we show that a general homogeneous polynomial $p\in \bC[x_0,x_1,...,x_n]$ of degree divisible by $k\ge 2$ can  be  represented  as a sum of at most $k^{n}$ $k$-th powers of homogeneous polynomials in $\bC[x_0, x_1,...,x_n]$. Noticeably, $k^{n}$   coincides with the number obtained by naive dimension count.
\end{abstract}

\maketitle

\section{Introduction}
\label{sec:int}
 We shall study  a version of the  general  Waring problem for rings as posed  in, e.g.,  \cite{GV}. Namely,   we shall be concerned  with 

\begin{problem} 
For any ring $A$ and any integer $k>1$, let $A_k \subset A$ be the set of all sums of $k$th powers in $A$. For
any $a \in A_k$, let $w_k(a, A)$ be the least $s$ such that $a$ is a sum of $s$ $k$-th powers. Determine $w_k(A)=\supp_{a\in A_k}w_k(a)$.  (It is possible that  $w_k(A)=\infty$). 
\end{problem} 

In many rings it makes sense to talk about generic elements in $A_k$ and, similarly, one can ask to determine the number 
$\w_k(A)=\supp_{a\in \widetilde A_k}w_k(a),$ 
where $ \widetilde A_k$ is the appropriate set of generic elements in $A_k$. We will refer to the latter question as the {\em weak Waring problem} as opposed to Problem 1 which we call the {\em strong Waring problem}. 

Below we concentrate on $A=\bC[x_0, x_1,...,x_n]$ and  for convenience work with homogeneous polynomials usually referred to as forms. In this case it is  known that $A_k$ coincides with the space of all forms in $\bC[x_0, x_1,...,x_n]$ whose degree is divisible by $k$.  
Thus, the strong Waring problem for  $\bC[x_0, x_1,...,x_n]$ is formulated as  follows. Denote by $S^d_n$  the linear space of all forms of degree $d$ in $n+1$ variables (with the $0$-form included).

\begin{problem}\label{prob:1} Find the supremum  
over the set of all forms  $f\in S^{kd}_n$  
of the minimal number of forms of  degree  $d$  needed to represent $f$ as a sum of their $k$-th powers.  
In particular, how many forms of degree  $d$ is required to represent an arbitrary form $f\in S^{2d}_n$  as a sum of their squares?  
\end{problem}

Recall that  $\dim S^d_n=\binom {d+n}{n}$ and simple  calculations show that  
$$\frac{\dim S^{kd}_n}{\dim S_{n}^d}< k^{n}\quad\text{and}\quad 
{\lim_{d\to\infty}\frac{\dim S_{n}^{kd}}{\dim S_{n}^d}=k^{n}.}$$ 
Therefore,  $k^{n}$ is the lower bound for the answer to Problem~\ref{prob:1}. A version of Problem~\ref{prob:1} related to the weak Waring problem is as follows. 

\begin{problem}\label{prob:generic} Find the minimum over all Zariski open subsets  in $S^{kd}_n$ 
of the  number of forms  of  degree  $d$  needed to represent forms from these subsets as a sum of their $k$-th powers.
In other words, how many $k$-th powers of forms of degree $d$ is required to present a general form of degree $kd$?
\end{problem} 

 For  sums of powers of linear forms a question very similar to Problem~\ref{prob:generic}   was studied in greater detail by J.~Alexander and A.~Hirschowitz in mid 90's  and was completely solved in  a series of papers  
culminated in \cite {AH}, see also \cite{Ci}, \cite {BO}. 
{(This problem  has a long history starting from the XIX century, see \cite{Iar,BO} and it was later posed anew by H.~Davenport.)} 
In our notation this means that one fixes  $d=1$ instead of letting $d$ be an arbitrary positive integer and uses $k$ as a parameter.  The above mentioned authors proved that the weak Waring problem for powers of linear forms 
{has the solution expected
 by  naive dimension count in all cases except for the case of quadrics in all dimensions, cubics in $5$ variables   and quartics in $3$, $4$ and $5$ variables.}  On the other hand, their results and further investigations {indicate} that for $n>1$  the number of powers of linear terms required to present an arbitrary form of a given degree  almost always exceeds the expected one obtained by  naive dimension count, see e.g., 1.6. of \cite{RSch}. 
 
 \medskip 
   Our main result is the following.

 \begin{theorem}\label{th:Nvar} Given a positive integer $k\ge 2,$ then  any general form $f$ of  degree $kd$ in $n+1$ variables is a sum of at most $k^{n}$ {$k$-th powers}. 
Moreover,  for a fixed $n$  this bound is sharp for all sufficiently large $d$.
 \end{theorem}
 
 Thus   $k^n$ gives an upper bound for the answer to Problem~\ref{prob:generic}  for any $n\ge 1$ and $k\ge 2$,
{and it is  optimal  for all sufficiently large $d$, see Remark $1$ in \S~3}. 
 
\medskip 
\noindent 
{\em Acknowledgments.} The second author wants to thank the Mittag-Leffler Institute for  hospitality and support during his visit to Stockholm in Spring 2011 when his collaboration with the other authors started. 
({The second author is a member of GNSAGA of Italian INDAM.})
The  third author  is sincerely grateful  to Professor Claus Scheiderer for the formulation of the problem and discussions as well as to the Department of Mathematics, University of Konstanz for the hospitality in December 2008 and March 2010.  
   
\section{Geometric reformulation and proof}

For simplicity we work  over $\bC$,  
although our results hold for any algebraically closed field of characteristic zero. 
We refer to \cite{Land} as a basic source of information on the geometry of tensors and its applications. 
The following result is classical,  see, e.g., \cite{CLM}.

 \begin{theorem}\label{th:2var} 

{\rm(i)} Any form $f$ of even degree $2d$ in $2$ variables is a sum of at most two squares; 

\noindent
{\rm(ii)} a general form  of even degree $2d$ in $2$ variables can be represented as a sum of two squares in exactly ${{2d-1}\choose d}$ ways.
 \end{theorem}
 
The proof follows from the identity
$$f=A\cdot B =\left[\frac{1}{2}(A+B)\right]^2+\left[\frac{i}{2}(A-B)\right]^2$$
and ${{2d-1}\choose d}=\frac{1}{2}{{2d}\choose d}$ is the number of ways $f$ can be presented as the product of two factors $A$ and $B$ of equal degree. Thus,  for $n=1$  and $k=2$ the answer to Problem~\ref{prob:1} is two.

We recall that for any projective variety $X$, its $p$-th secant variety is defined as 
 the Zariski closure of the union of the projective spans 
$<x_1,\ldots , x_p>$ where $x_i\in X$.
The following result gives a convenient reformulation of our problem.

\begin{theorem}\label{th:3} Given a linear space $V,$ a general polynomial in $S^{kd}V$ is a sum
of $p$ $k$-th powers $g_1^k,\ldots g_p^k$ where $g_i\in S^d V$
if and only if for $p$ general  forms $g_i\in S^d V$, $i=1,\ldots p$,
the ideal generated by {$g_1^{k-1},\ldots g_p^{k-1}$}   contains $S^{kd}V$. (We shall call such an ideal  $kd$-regular.) 
\end{theorem}
\begin{proof} The statement is a direct consequence of Terracini's lemma.
 Consider the subvariety
$X$ in the ambient space $\PP S^{kd}V$ consisting of the $k$-th powers of all forms from $S^{d}V$. 
The tangent space to $X$ at $g_i^{k}\in X$ is of the form $\{g_i^{k-1}f|f\in S^d V\}$.
Therefore,  the $p$-secant variety of $X$ coincides with  the ambient space $\PP S^{kd}V$ if and only if 
the span of the tangent spaces to $X$ at general $g_i^k$,
(which is equal to $\{\sum_{i=1}^pg_i^{k-1}f_i|f_i\in S^dV\}$), coincides with  $\PP S^{kd}V$ as well.
\end{proof}

Theorem~\ref{th:3}  relates Problem~\ref{prob:generic} to a special case
of a conjecture of the first author 
about the Hilbert series of  ideals generated by general forms in given degrees, see  \cite{FH}.

\medskip
We will show that  if $V$ is an $(n+1)$-dimensional linear space   then the ideal generated by $k^n$ general forms of the form
$g_i^{k-1}$ where $g_i\in S^{d}_nV$ is $kd$-regular, i.e., contains $S^{kd}V$. 

In order to do this  it suffices to find $k^n$
 specific polynomials $\{g_1,\ldots g_{k^n}\}$ of degree $d$ such that the ideal generated by the powers $g_i^{k-1}$ is $kd$-regular. Below, we will choose as  $g_i$'s  powers of certain linear forms.
 For powers of linear forms one can use  a new point of view related to apolarity. 
The  space $T_{g^k}X^{\perp}$ orthogonal to $T_{g^k}X=\{g^{k-1}f|f\in S^d V\}$ is given by 
$T_{g^k}X^{\perp}=\{h\in S^{kd}V^{\vee}| h\cdot g^{k-1}=0\in S^dV^{\vee}\}$, i.e., is 
 the space of polynomials in $V^{\vee}$  apolar to $g^{k-1}$. 
Moreover, when $g=l^m$, $l\in V$ the classical theory of apolarity provides a better result (for a recent reference see  Lemma on page 1094 of \cite{Iar}).

\begin{proposition}\label{pr:1} A form $f\in S^mV^{\vee}$ is apolar to {$l^{m-k}$}, i.e., $l^{m-k}f=0$ if and only if   all the derivatives of $f$ of order $\le k$ vanish at {$l\in V$}. 
\end{proposition}

Using Proposition~\ref{pr:1} one can reduce   Theorem \ref{th:Nvar} to the following statement.

\begin{theorem}\label{greform}
For a given  integer $k\ge 2$   a form of degree $kd$ in $(n+1)$ variables
which has all derivatives of order $\le d$ vanishing at $k^n$ general points
vanishes identically.
 \end{theorem}

Our final effort will be  to settle Theorem \ref{greform}.
Denote   
by $x_0,\ldots x_n$ a basis of $V$.
Let $\xi_i=e^{2\pi i\sqrt{-1}/k}$ for $i=0,\ldots, k-1$ be the (set of all) $k$-th roots of unity.
 By semicontinuity, it is enough to find  $k^n$ special points  
in ${\PP V}\simeq {\bf P}^n$ such that a polynomial of degree $kd$ in ${\bf P}^n$
which has all derivatives of order $\le d$ vanishing at these points must necessarily vanish  identically. As such points we 
 choose the points 
$(1,\xi_{i_1},\xi_{i_2},\ldots,\xi_{i_n})$ where $0\le i_j\le k-1$,
$1\le j\le n$.

\medskip 
The following result proves even more than was claimed in  Theorem \ref{greform}.

\begin{theorem}\label{greform1}
For a given integer  $k\ge 2$   a form of degree $kd+k-1$ in $(n+1)$ variables
which has all derivatives of order $\le d$ vanishing at $k^n$ general points
vanishes identically. 
 \end{theorem}
\begin{proof}

As above we choose as  our configuration the $k^n$ points 
$(1,\xi_{i_1},\xi_{i_2},\ldots,\xi_{i_n})$ where $0\le i_j\le k-1$,
$1\le j\le n$. Consider first the  case  $n=1$.
If a form $f(x_0,x_1)$ of degree $kd+k-1$ has its 
derivatives of order $\le d$ vanishing at all $(1,\xi_i)$, then 
$f$ should be {divisible by} $(x_1-\xi_ix_0)^{d+1}$ 
for $i=0,\ldots k-1$ and, therefore,  {if $f$ is not vanishing identically, 
then its  degree should be at least $k(d+1)$, which is a contradiction.}

For $n\ge 2$ consider the {arrangement} of  ${n\choose 2}k$ hyperplanes given by  $x_i=\xi_s x_j$ where 
$1\le i<j\le n$, $0\le s\le k-1$. One can easily check that this {arrangement}   has the property that each hyperplane contains exactly $k^{n-1}$ points and,  furthermore, each point is contained in exactly $n\choose 2$ hyperplanes. 
{Indeed, consider, for example, the hyperplane $\HH$ given by $x_n=\xi_ix_{n-1}$.
The natural parametrization of $\HH$ is by 
 $(x_0,\ldots, x_{n-1})\mapsto 
(x_0,x_1,\ldots,x_{n-1},\xi_ix_{n-1})$ and the $k^{n-1}$ points which lie on $\HH$
correspond, according to this parametrization, exactly to   $(1,\xi_{i_1},\xi_{i_2},\ldots,\xi_{i_{n-1}})$ for $0\le i_j\le k-1$,
$1\le j\le n-1$. In other words, they correspond exactly to our arrangement of points in the previous dimension $n$.} Our  proof now proceeds  by a double induction on the number of variables $n$ and degree $d$.
Assume that the statement holds  for all $d$ and up to  $n$ variables. (The case $n=1$ is settled above.) Let us perform a step of  induction in $d$. First {we settle the case $d\le {n\choose 2}-1$. Consider a polynomial $f$ of degree $kd+k-1$ satisfying our assumptions.
Restricting  $f$ to each of  the above ${n\choose 2}k$ hyperplanes  $x_i=\xi_s x_j$ where 
$1\le i<j\le n$, $0\le s\le k-1$ we obtain the same situation in dimension $n$.
By the induction hypothesis $f$ vanishes on each such hyperplane and, therefore,  must be divisible by  $H$, where $H$ is   the product of the linear forms $x_i=\xi_sx_j$ defining all the chosen hyperplanes. (Obviously, $\deg H= {n\choose 2}k$.) 
Thus, $f$ vanishes identically since $k\left({n\choose 2}-1\right)+k-1<{n\choose 2}k$.} 
 For higher degrees we argue as follows. Take a form  $f$  of degree {$kd+k-1$} satisfying our assumptions.
Restricting  as above $f$ to each of  the above ${n\choose 2}k$ hyperplanes  $x_i=\xi_s x_j$  we obtain the same situation in dimension $n$.
{Again}, by the induction hypothesis $f$ vanishes on each such hyperplane and must be divisible by  $H$. We get
$$f=H\tilde f$$
where {$\deg \tilde f=k\left(d-{n\choose 2}\right)+k-1$} and
$\tilde f$ has all derivatives of order $\le d-{n\choose 2}$ vanishing at the same $k^n$
points $(1,\xi_{i_1},\xi_{i_2},\ldots,\xi_{i_n})$. {Indeed,
in any affine coordinate system centered at any of these points, $f$ has no terms of degree $\le d$. 
Since $H$ has its lowest term in degree $n\choose 2$
it follows that $\tilde f$ has no terms in degree $\le d-{n\choose 2}$.}
By the induction hypothesis $\tilde f$ is identically zero. 
\end{proof}

\medskip
Notice that  we have also obtained the following result of  independent interest.
\medskip 
\begin{corollary}Any form of degree $kd$ 
in $(n+1)$ variables can be expressed as a linear combination
 of the polynomials
$(x_0+\xi_{i_1}x_1+\xi_{i_2}x_2+\ldots +\xi_{i_n}x_n)^{(k-1)d}$ with coefficients being  polynomials of degree $d$. 
\end{corollary}

 \section{Final remarks}
 
 \noindent 
 {\bf Remark 1.} Although $k^n$ is the correct asymptotic bound,
it seems to be  sharp only for considerably large values $d$. In particular,  computer experiments show that for $k=2,\; n=3$ 
and $d\le 20$ seven general polynomials of degree $d$ suffice to generate the space of polynomials in degree $2d$.   All eight polynomials are required  only for $d\ge 21$ . Similarly, for $n=4$ and $d\le 75$  experiments suggest that $15$ (instead of the expected $16$)
general polynomials of degree $d$ suffice to generate the space of polynomials  in degree $2d$. Analogously,  all $16$ polynomials are required  for $d\ge 76$. The ultimate challenge of this project is to solve completely Problem~\ref{prob:generic} for triples $(n,k,d)$, and, in particular,  to find the complete list of exceptional triples for which the answer to Problem~\ref{prob:generic}  is larger than the one obtained by dimension count. Obviously, this list should include the list of exceptional cases obtained earlier by J.~Alexander and A.~Hirschowitz.  

\medskip
\noindent
{\bf Remark 2.} Theorem~\ref{th:Nvar}  seems to be new even in the classical case $k=2$, i.e.,  
for a sum of squares. In this case we have shown that any form of degree $2d$ 
in $(n+1)$ variables can be expressed as a linear combination of the polynomials
$(x_0\pm x_1\pm x_2+\ldots +\pm x_n)^{d}$  with coefficients being polynomials of degree $d$. Note that the former polynomials have real coefficients.   But, obviously,  our main result does not hold over the reals. It only implies that there is, in the usual topology, 
an open  set  of real polynomials of degree $2d$
which can be expressed as real linear combinations of $2^n$ squares of
real polynomials of degree $d$. In other words,  $2^n$ is a typical rank, see e.g., \cite{CO}.  
Notice that  other typical ranks might also appear on other open subsets {of} polynomials. 
An example with three distinct typical ranks occurring in a similar situation can be found in  \cite{CO}.  

\medskip
\noindent
{\bf Remark 3.} Although we only used  powers of linear forms as the generators of the ideal in the above arguments 
it is not true that a general polynomial of degree $kd$
can be expressed as a sum of at most $k^n$ powers $l_i^{kd}$ of linear forms $l_i$. 
For large $d$ the number of necessary summands of the latter problem (solved by  J.~Alexander and A.~Hirschowitz) 
equals 
$\left\lceil\frac{{{kd+n}\choose n}}{n+1}\right\rceil$.  It grows as
$(kd)^n/(n+1)!$ and  is considerably larger than $k^n$. 

\medskip
\noindent
{\bf Remark 4.} Notice that the family of ideals generated by the powers of linear forms $(\xi_0x_0+\xi_1 x_1+\xi_2x_2+\ldots +\xi_n x_n)^{d}$ 
is a special case of ideals associated with hyperplane arrangements that appeared in several publications in the last decade, see e.g., \cite{PSh} and \cite{AP}. In particular, it should be possible to calculate the Hilbert series of the quotient of the polynomial ring modulo these ideals and the answer should be a certain specialization of the Tutte polynomial of the vector configuration given by the above linear forms, cf. \S~5 of \cite {AP}.  

\medskip 
\noindent
{\bf Remark 5.} The above mentioned  conjecture of the first author prescribes the Hilbert series of a homogeneous ideal generated by general forms of given degrees. Computer experiments show that the ideals generated by the powers of linear forms $(\xi_0x_0+\xi_1 x_1+\xi_2x_2+\ldots +\xi_n x_n)^{d}$  have, in general, another Hilbert series. On the other hand, it seems that in case $k=2$ a different family  of ideals generated by the powers of  linear forms have the predicted Hilbert series. Namely, for every non-empty subset $I\subset\{0,\ldots n\}$ define $x_I=\sum_{i\in I}x_i$ and  take  $(x_I)^d$ for all subsets $I$ with $|I|$ odd as generators of the ideal in question.  

\medskip 
\noindent
{\bf Remark 6.} It is classically known that plane quartics can be represented as a sums of three squares. It was recently observed in \cite{BHORS} that the {closure of the} set of plane sextics which are a sums of three squares forms  a hypersurface  
 of degree 83200 in the space of all sextics. 

{
\medskip 
\noindent
{\bf Remark 7.} Our results can be interpreted in the setting of osculating varieties.
In notations of \cite{BCGI} we have shown that any $k$-osculating space at the $k^n$-th secant variety
of the $kd$-Veronese embedding of $\PP^n$ fills out the ambient space.}


\begin{thebibliography}{30}


\bibitem {AH} J.~Alexander, A.~Hirschowitz, {Polynomial interpolation in several variables}, J.Alg. Geom., {\bf 4} (1995),  201--222. 

\bibitem{AP}ÊF.~Ardila,  A.~Postnikov,  Combinatorics and geometry of power ideals. Trans. Amer. Math. Soc. {\bf 362(8)}  (2010),  4357--4384.

\bibitem{CLM} M.~Choi, T.~Lam, B.~Reznick, {Sums of squares of real polynomials}, 
Proc. of Symp. in Pure Math. {\bf 58(2)} (1995), 103--126.

\bibitem{GV} L.~Gallardo, L.~Vaserstein, {The strict Waring problem for polynomial rings}, Journal of Number Theory {\bf 128} (2008), 2963--2972. 

\bibitem {Iar} A.~Iarrobino, {Inverse systems of a symbolic power II. The Waring problem for forms},  J. Algebra, {\bf 174} (1995), 1091--1110.  

\bibitem{BCGI} A.~Bernardi, M.~V.~Catalisano, A.~Gimigliano, M.~Id\'a, Secant varieties to osculating varieties of Veronese embeddings of $\PP^n$, 
J. Algebra, {\bf 321} (2009),  982--1004. 

\bibitem{BHORS} G.~Blekherman, J.~Hauenstein, J.~C.~Ottem, K.~Ranestad, B.~Sturmfels, Algebraic boundaries of Hilbert's SOS cones, arXiv:1107.1846. 

\bibitem{BO} M.~C.~Brambilla, G.~Ottaviani, {On the Alexander-Hirschowitz theorem}, J. Pure Appl. Alg. {\bf 212} (2008), 1229--1251. 

\bibitem {Ci}  C.~Ciliberto,  {Geometric Aspects of Polynomial Interpolation in More Variables and of Waring's Problem},  European Congress of Mathematics, Vol 1 (Barcelona 2000), 289--316, Progr. Math., 201, Birkh\"auser, Basel, 2001.

\bibitem{CO} P.~Comon, G.~Ottaviani, On the typical rank of real binary form,  Linear  Multilinear  Alg., to appear.

\bibitem{FH} R.~Fr\"oberg,  {An inequality for Hilbert series of graded algebras,} Math. Scand.  {\bf 56} (1985),  117--144.   

\bibitem{Land} J.~M.~Landsberg, Tensors: Geometry and Applications, Graduate Studies in Mathematics, {\bf 128}, (2012) AMS, Providence, RI, 439 pp. 

\bibitem{PSh} A.~Postnikov, B.~Shapiro,  Trees, parking functions, syzygies, and deformations of monomial ideals. Trans. Amer. Math. Soc. {\bf 356(8)} (2004),  3109--314. 

\bibitem{RSch} K.~Ranestad, F.~O.~Schreyer, {Varieties of sums of powers}, J. Reine Angew. Math. {\bf 525} (2000), 147--181. 



\end{thebibliography}
\end{document}